\documentclass[a4paper]{amsart}
\usepackage[12pt]{extsizes}

\usepackage[english]{babel}
\usepackage{amsfonts,amsmath,amsthm,amssymb,latexsym,mathrsfs,enumitem}
\usepackage{geometry}
\newtheorem{theorem}{Theorem}[section]

\newtheorem{lemma}[theorem]{Lemma}
\newtheorem{proposition}[theorem]{Proposition}

\theoremstyle{plain}

\newtheorem{remark}[theorem]{Remark}

\newtheorem{thmm}{Theorem}

\begin{document}

\title[Extension of fragmented Baire-one functions]{Extension of fragmented Baire-one functions \\ on Lindel\"{o}f spaces}
\author{Olena Karlova}
\author{Volodymyr Mykhaylyuk}

\small{{\it Chernivtsi National University, Ukraine} \\ {\it Jan Kochanowski University in Kielce, Poland}}

\begin{abstract}
  We investigate the possibility of extension of fragmented functions from Lindel\"{o}f subspaces of
completely regular spaces and find necessary and sufficient conditions on a fragmented Baire-one function to be extendable on any completely regular superspace.
\end{abstract}

\maketitle

\section{Introduction}

It is well known that a Baire-one function ($=$ a pointwise limit of a sequence of continuous functions) on a $G_\delta$-subset of a metric space can be extended to a Baire-one function defined on the whole space (see \cite[\S 35, VI]{Ku2}). In 2005 O.~Kalenda and J.~Spurn\'{y} \cite{Kal_Spu} obtained the following result.

\begin{thmm}\label{th:KS}
  Let $E$ be a Lindel\"{o}f hereditarily Baire subset of a completely regular space $X$
and $f:E\to\mathbb R$ be a Baire-one function. Then there exists a Baire-one function
$g:X\to\mathbb R$ such that $f = g$ on $E$.
\end{thmm}
Let us observe that Theorem~\ref{th:KS} gives new results even in case of separable
metric spaces. For example, if $B\subseteq\mathbb R$ is the Bernstein set, then it is a hereditarily Baire Lindel\"{o}f
space which is Borel non-measurable. Nevertheless any Baire-one function on $B$
 can be extended to a Baire-one function on $\mathbb R$ by Theorem~\ref{th:KS}.

It is easy to see that the  assumption that $E$ is hereditarily Baire cannot be omitted in Theorem~\ref{th:KS}. Indeed,  if $A$ and $B$ are disjoint dense subsets of $E=\mathbb Q\cap [0,1]$ such that $E=A\cup B$ and $X=[0,1]$ or $X=\beta E$, then the characteristic function $f=\chi_A:E\to [0,1]$, being a Baire-one function, can not be extended to a Baire-one function on $X$. The reason is that $\chi_A$ is not a fragmented function.

Recall~\cite{JR:1982} that for some $\varepsilon>0$ a function $f:X\to (Y,d)$ from a topological space to a metric one is said to be {\it $\varepsilon$-fragmented}, if for every  closed nonempty set  $F\subseteq X$ there exists a nonempty relatively open set $U\subseteq F$ such that ${\rm diam}f(U)<\varepsilon$. If $f$ is $\varepsilon$-fragmented for every  $\varepsilon>0$, then it is called  {\it fragmented}.

If $X$ is hereditarily Baire, then every Baire-one map $f:X\to (Y,d)$ is barely continuous (i.e., for every nonempty closed set $F\subseteq X$ the restriction $f|_F$ has a point of continuity) and, hence, is fragmented (see \cite[31.X]{Ku2}). If $X$ is a  paracompact space in which every closed set is $G_\delta$, then every fragmented map $f:X\to (Y,d)$ is Baire-one in the case either ${\rm dim} X=0$, or $Y$ is a  contractible locally path-connected space~\cite{Karlova:Mykhaylyuk:Comp,Karlova:Mykhaylyuk:2016:EJMA}.

Therefore, it is natural to ask if any fragmented Baire-one function $f:E\to\mathbb R$ defined on a Lindel\"{o}f subspace $E$ of a completely regular space $X$ can be extended to a Baire-one function on the whole space $X$? The following theorem is our main result (see Theorem~\ref{thm:equiv:main}).
\begin{thmm}
   Let $X$ be a Lindel\"{o}ff space and $f:X\to\mathbb R$ be a Baire-one function. Then following conditions are equivalent: 1) $f$ is extendable to a Baire-one function on any completely regular superspace $Y\supseteq X$; 2) $f$ is extendable to a Baire-one function on any compactification $Y$ of $X$; 3) $f$ is extendable to a Baire-one function on $\beta X$; 4) $f$ is functionally countably fragmented; 5) $f$ is fragmented.
\end{thmm}
In the last section we give an example of a completely metrizable locally compact space $X$ and a Baire one function $f:X\to[0,1]$ such that $f$ is not countably fragmented, in particular, $f$ can not be extended to a Baire one function $g:\beta X\to[0,1]$. This gives a negative answer to a question of Kalenda and Spurn\'{y} \cite[Question 1]{Kal_Spu} (see also \cite[Theorem 7]{Karlova:Mykhaylyuk:2017:ExtensionB1}).

\section{Properties of fragmented maps and resolvable sets}

A subset $A$ of a topological space $X$ is an {\it $H$-set} or {\it resolvable in the sense of Hausdorff}, if there exists a  decreasing sequence $(F_\xi)_{\xi\in[0,\alpha)}$ of closed subsets of $X$ such that
\begin{gather*}
   A=\bigcup_{\xi<\alpha,\,\xi {\footnotesize \mbox{\,\,is odd}}}(F_\xi\setminus F_{\xi+1}).
\end{gather*}
If each set $F_\xi$ can be chosen to be functionally closed, we say that $A$ is {\it functionally $H$-set} or {\it functionally resolvable}. Moreover, if $|\alpha|\le\aleph_0$, then $A$ is called {\it functionally countably  $H$-set} or {\it functionally countably resolvable}.

It is well-known that $A$ is an $H$-set if and only if for any nonempty closed set $F\subseteq X$ there exists a relatively open set $U\subseteq F$ such that either $U\subseteq A$ or $U\subseteq X\setminus A$ (see \cite[\S 12]{Ku2}).

Let $\mathscr U=(U_\xi:\xi\in[0,\alpha])$ be  a transfinite sequence of subsets of a topological space $X$. Following~\cite{HolSpu}, we define $\mathscr U$ to be {\it regular in $X$}, if
\begin{enumerate}[label=(\alph*)]
  \item each $U_\xi$ is open in $X$;

  \item $\emptyset=U_0\subset U_1\subset U_2\subset\dots\subset U_\alpha=X$;

  \item\label{it:c} $U_\gamma=\bigcup_{\xi<\gamma} U_\xi$ for every limit ordinal $\gamma\in[0,\alpha)$.
\end{enumerate}

It was proved in \cite{Karlova:Mykhaylyuk:2017:ExtensionB1} that   a map $f:X\to Y$ is $\varepsilon$-fragmented if and only if  there exists a regular sequence $\mathscr U=(U_\xi:\xi\in[0,\alpha])$ (which is called {\it $\varepsilon$-associated with $f$} and is denoted by $\mathscr U_\varepsilon(f)$) in $X$ such that ${\rm diam} f(U_{\xi+1}\setminus U_\xi)<\varepsilon$ for all $\xi\in[0,\alpha)$.

We say that an $\varepsilon$-fragmented  map $f:X\to Y$ is {\it functionally $\varepsilon$-fragmented} if $\mathscr U_\varepsilon(f)$ can be chosen such that every set $U_\xi$ is functionally open in $X$. Further, $f$ is {\it functionally $\varepsilon$-countably fragmented} if $\mathscr U_\varepsilon(f)$ can be chosen to be countable and $f$ is {\it functionally countably fragmented} if $f$ is functionally $\varepsilon$-countably fragmented for all $\varepsilon>0$.

A set $A$ is called {\it functionally $G_\delta$-set (functionally $F_\sigma$-set)}, if $A$ is an intersection (a union) of a sequence of functionally open (functionally closed) sets. Further, a set $A$ is said to be {\it functionally ambiguous}, if it is functionally $G_\delta$ and functionally $F_\sigma$ simultaneously.

\begin{proposition}\label{prop:properties} Let $X$ be a topological space and $(Y,d)$ be a metric space.
\begin{enumerate}
\item Every functionally countably resolvable set $A\subseteq X$ is functionally ambiguous.

\item The class of all functionally countably resolvable subsets of $X$ is closed under finite unions and intersections.

\item A finite-valued function $f:X\to Y$ is functionally countably fragmented if and only if the set $f^{-1}(y)$ is functionally countably resolvable for every $y\in f(X)$.

\item The class of all functionally countably fragmented functions between $X$ and $Y$ is closed under uniform limits.\label{prop:properties:it:4}

\item   Every functionally countably fragmented function $f:X\to Y$ is functionally $F_\sigma$-measurable.\label{prop:properties:it:5}
\end{enumerate}
\end{proposition}

\begin{proof}  The properties 1)--3)  follow straightforwardly  from the definitions and we omit their proofs.

4). Let $(f_n)_{n\in\omega}$ be a sequence of functionally countably fragmented functions $f_n:X\to Y$ which is convergent uniformly to a function $f:X\to Y$.
Fix $\varepsilon>0$ and choose a number $k\in\omega$ such that $d(f_n(x),f(x))<\tfrac{\varepsilon}{4}$ for all $n\ge k$ and $x\in X$. Since $f_k$ is functionally $\varepsilon/4$-countably fragmented, one can fix a countable $\varepsilon/4$-associated sequence $\mathscr U=(U_\xi:\xi\in[0,\alpha])$ with $f_k$. Then for all $\xi\in [0,\alpha]$ and for all $x,y\in U_{\xi+1}\setminus U_\xi$ we have
\begin{gather*}
  d(f(x),f(y))\le d(f(x),f_k(x))+d(f_k(x),f_k(y))+d(f_k(y),f(y))<\frac{\varepsilon}{4}+\frac{\varepsilon}{4}+\frac{\varepsilon}{4}=\frac{3\varepsilon}{4}.
\end{gather*}
Hence, ${\rm diam}f(U_{\xi+1}\setminus U_\xi)\le \frac{3\varepsilon}{4}<\varepsilon$.

5). For every $n\in\mathbb N$ we take a sequence $\mathscr U_{n}=(U_\xi:\xi\in[0,\alpha])$ of functionally open sets which is $\frac 1n$-associated with $f$. Now let $G$ be an open set in $Y$ and $x\in f^{-1}(G)$. Take $n\in\mathbb N$ such that the open ball  $B(f(x),\frac 1n)$ with center at $f(x)$ and radius $\frac 1n$ is contained in $G$. Let $\xi\in[0,\alpha]$ be such that $x\in U_{\xi+1}\setminus U_\xi$. Then $U_{\xi+1}\setminus U_\xi\subseteq f^{-1}(G)$. Therefore, there exists a subfamily $\mathscr V\subseteq\mathscr U_n$ such that $f^{-1}(G)=\cup \{V:V\in\mathscr V\}$. Sine $\mathscr V$ is at most countable family of functionally $F_\sigma$-sets, $f^{-1}(G)$ is a functionally $F_\sigma$-set in $X$.

\end{proof}

\section{Separation theorem for functionally $G_\delta$-sets}

The proposition below is tightly connected with \cite[Proposition 11]{Kal_Spu} and we use  ideas from \cite{Kal_Spu} for the proof.

\begin{proposition}\label{prop:extGdeltaInBetaX}
 Let $A$ and $B$ be functionally ambiguous disjoint subsets of a Lindel\"{o}ff space $X$ such that $\overline{F\cap A}\ne F$ or $\overline{F\cap B}\ne F$ for every nonempty closed set $F\subseteq X$. Then there exist disjoint functionally $G_\delta$-sets $\widetilde A$ and $\widetilde B$ in $\beta X$ such that  $\widetilde A\cap X=A$ and $\widetilde B\cap X=B$.
\end{proposition}

\begin{proof} Since $X$ is $C^*$-embedded in $\beta X$, there exist functionally $G_\delta$-sets $A_1$ and $B_1$ in $\beta X$ such that $A_1\cap X=A$ and $B_1\cap X=B$. We put $C=A_1\cap B_1$. Clearly, we need to consider the case $C\ne\emptyset$ only.

 Let $D$ stands for the set of all points $x\in C$ for which there exist functionally $F_\sigma$-sets $U_x, V_x\subseteq \beta X$ such that  the set $W_x=(U_x\cup V_x)\cap C$ is a neighborhood of $x$ in $C$ and $U_x\cap A=\emptyset=V_x\cap B$.

 We prove that $D=C$. Suppose to the contrary that $C\setminus D\ne\emptyset$ and put $F=\overline{C\setminus D}$. The conditions of the proposition imply that $\overline{F\cap A}\ne F$ or $\overline{F\cap B}\ne F$. For definiteness we assume that $\overline{F\cap A}\ne F$. Then there exist a point $x_0\in C\setminus D$ and a functionally open in $Z$ neighborhood $U$ of $x_0$ such that $U\cap F\cap A=\emptyset$. Note that the set $A_1=A\cap U$ is Lindel\"{o}ff as an $F_\sigma$-subset of the Lindel\"{o}ff space $X$. Since $A_1\cap F=\emptyset$, there exists a functionally open set $V\supseteq A_1$ in $\beta X$ such that $V\cap F=\emptyset$. Therefore, $V\cap(C\setminus D)=\emptyset$ and $D_1=V\cap C\subseteq D$. According to \cite[Proposition 5]{Kal_Spu}, the set $D_1$ is Lindel\"{o}ff. Hence, there exists an countable set $E\subseteq D_1$ such that $D_1 \subseteq \bigcup_{d\in E}W_d$. Now we put
 \begin{gather*}
 U_{x_0}=(U\setminus V)\cup\bigcup_{d\in E}U_d \quad\mbox{and}\quad V_{x_0}=\bigcup_{d\in E}V_d.
 \end{gather*}
Notice that $U_{x_0}$ and $V_{x_0}$ are functionally $F_\sigma$-sets in $\beta X$, $U_{x_0}\cap A=\emptyset=V_{x_0}\cap B$ and
$$
U\cap C \subseteq (U\setminus V)\cup (V\cap C)=(U\setminus V)\cup D_1\subseteq U_{x_0}\cup V_{x_0}.
$$
It follows that $x_0\in D$, which contradicts to the choice of $x_0$.

Since $C$ is Lindel\"{o}ff, there exists a countable set $\{c_n:n\in\omega\}\subseteq C$ such that $C\subseteq\bigcup_{n\in\omega}W_{c_n}$. It remains to put
\begin{gather*}
\widetilde A=A_1\setminus \left(\bigcup_{n\in\omega}U_{c_n}\right)\quad\mbox{and}\quad\widetilde B=B_1\setminus \left(\bigcup_{n\in\omega}V_{c_n}\right).
\end{gather*}
\end{proof}

\begin{proposition}\label{prop:SeparationTh}
  Let $A$ and $B$ be functionally ambiguous disjoint subsets of a Lindel\"{o}ff space $X$ such that $\overline{F\cap A}\ne F$ or $\overline{F\cap B}\ne F$ for every nonempty closed set $F\subseteq X$. Then there exists a  functionally countably resolvable set $C$ in $X$ such that $A\subseteq C\subseteq X\setminus B$.
\end{proposition}

\begin{proof}
  By Proposition~\ref{prop:extGdeltaInBetaX} there are disjoint functionally $G_\delta$-sets $\widetilde A$ and $\widetilde B$ in $\beta X$ such that  $\widetilde A\cap X=A$ and $\widetilde B\cap X=B$.  Applying  functional version of Sierpi\'{n}ski's Separation Theorem~(see \cite[Lemma 4.2]{Karlova:CMUC:2013} and \cite[p.~350]{Ku2}) we obtain functionally ambiguous set $\widetilde C\subseteq \beta X$ such that $\widetilde A\subseteq \widetilde C\subseteq \beta X\setminus \widetilde B$. It is easy to see that $f=\chi_{\widetilde C}\in{\rm B}_1(\beta X,\mathbb R)$. Then $f$ is functionally countably fragmented by \cite[Proposition 2]{Karlova:Mykhaylyuk:2017:ExtensionB1}. Therefore, the set $C=\widetilde C\cap X$ is functionally countably resolvable.
\end{proof}

\section{Extension of functionally countably fragmented maps with values in R-spaces}
A metric space $(Y,d)$ is said to be an {\it R-space}, if for every $\varepsilon>0$ there exists a continuous map   $r_{\varepsilon}:Y\times Y\to Y$ with the following properties:
\begin{gather*}
  d(y,z)\le\varepsilon\,\, \Longrightarrow\,\, r_\varepsilon(y,z)=y,\\
d(r_\varepsilon(y,z),z)\le\varepsilon
\end{gather*}
for all $y,z\in Y$.

It is worth noting that every convex subset $Y$ of a normed space $(Z,\|\cdot\|)$ equipped with the metric induced from $(Z,\|\cdot\|)$ is an R-space, where the map $r_{\varepsilon}$ is defined as
\begin{gather*}
r_{\varepsilon}(y,z)=\left\{\begin{array}{ll}
                              z+(\varepsilon/\|y-z\|)\cdot (y-z), & \|y-z\|>\varepsilon, \\
                              y & \mbox{otherwise}.
                           \end{array}
\right.
\end{gather*}

\begin{remark}\label{rem:local_conn_R_space}
{\rm It follows from the definition that every path-connected R-space is locally path-connected.}
\end{remark}

\begin{lemma}\label{lem:unif_conv_ext_R_space}
Let $X$ be a topological space, $E\subseteq X$, $(Y,d)$ be a metric R-space  and $f\in {\rm B}_1(E,Y)$. If there exists a sequence of functions $f_n\in {\rm B}_1(X,Y)$ such that $(f_n)_{n=1}^\infty$ converges uniformly to $f$ on $E$, then $f$ can be extended to a function $g\in {\rm B}_1(X,Y)$.
\end{lemma}

\begin{proof}  Let $(r_n)_{n=1}^\infty$ be a sequence of continuous functions $r_n:Y\times Y\to Y$ such that
\begin{gather}
  d(y,z)\le \frac{1}{2^n}\,\, \Longrightarrow\,\, r_n(y,z)=y,\label{gath:u2}\\
d(r_n(y,z),z)\le\frac{1}{2^n}
\end{gather}
for all $y,z\in Y$ and $n\in\omega$.

Without loss of generality we may assume that
\begin{gather}\label{gath:u1}
d(f_{n+1}(x),f_{n}(x))\le \frac{1}{2^n}
\end{gather}
for all $n\in\omega$ and $x\in E$.

For all $x\in X$ we define
\begin{gather*}
  g_1(x)=f_1(x),\\
  g_n(x)=r_{n-1}(f_n(x),g_{n-1}(x))\,\,\,\mbox{for all}\,\, n\ge 2.
\end{gather*}
Then every $g_n:X\to Y$ belongs to the first Baire class as a composition of continuous and Baire-one functions.

Let us observe that
\begin{gather*}
  d(g_{n+1}(x),g_{n}(x))=d(r_{n}(f_{n+1}(x),g_{n}(x)),g_n(x))\le\frac{1}{2^n}
\end{gather*}
for all $x\in X$ and $n\in\omega$. It follows that the sequence $(g_n)_{n=1}^\infty$ is uniformly convergent to a Baire-one map $g:X\to Y$ \cite[Theorem 4]{Karlova:Mykhaylyuk:2017:Limits}.

If $x\in E$, then (\ref{gath:u1}) and (\ref{gath:u2}) imply that
\begin{gather*}
  g_2(x)=r_1(f_2(x),g_1(x))=r_1(f_2(x),f_1(x))=f_2(x).
\end{gather*}
Assume that
\begin{gather*}
  g_k(x)=f_k(x)
\end{gather*}
for all $k\le n$. Then
\begin{gather*}
  g_{n+1}(x)=r_{n}(f_{n+1}(x),g_{n}(x)=r_n(f_{n+1}(x),f_n(x))=f_{n+1}(x).
\end{gather*}
Therefore,
\begin{gather*}
g(x)=\lim_{n\to\infty}g_n(x)=\lim_{n\to\infty}f_n(x)=f(x)
\end{gather*}
for all $x\in E$.
\end{proof}

Recall that a subspace $E$ of a topological space $X$ is {\it $z$-embedded} in $X$~\cite{GillJer} if for every functionally closed set $A$ in $E$ there exists a functionally closed set $B$ in $X$ such that $B\cap E=A$. It is well-known that each Lindel\"{o}ff subspace of a completely regular space  is $z$-embedded~\cite{BlairHager}.

The proof of the following result is quite similar to the proof of \cite[Proposition 4]{Karlova:Mykhaylyuk:2017:ExtensionB1}.

\begin{proposition}\label{prop:ext_countable_frag_ness_R}
  Let $E$ be a $z$-embedded subspace of a completely regular space $X$, $Y$ be a path-connected separable R-space and $f:E\to Y$ be a functionally countably fragmented  function. Then  $f$ can be extended to a functionally countably fragmented function $g\in {\rm B}_1(X,Y)$.
  \end{proposition}

  \begin{proof} Let us observe that we may assume the space $X$ to be compact. Indeed, $E$ is $z$-embedded in $\beta X$, since $X$ is $C^*$-embedded in $\beta X$, and if we can extend $f$ to a  functionally countably fragmented function $h\in {\rm B}_1(\beta X,Y)$, then the restriction $g=h|_E$ is a  functionally  countably fragmented extension of $f$ on $X$ and $g\in {\rm B}_1(X,Y)$.

  Fix $n\in\mathbb N$ and consider  $\frac 1n$-associated with $f$ sequence $\mathscr U=(U_\xi:\xi\leq\alpha)$. Without loss of the generality we can assume that all sets $U_{\xi+1}\setminus U_\xi$ are nonempty.
   Since $E$ is $z$-embedded in $X$, one can choose a countable family $\mathscr V=(V_\xi:\xi\leq\alpha)$ of functionally open sets in $X$ such that $V_{\xi}\subseteq V_{\eta}$ for all $\xi\le\eta\leq\alpha$, $V_{\xi}\cap E=U_{\xi}$ for every $\xi\leq\alpha$ and $V_\eta=\bigcup_{\xi<\eta}V_\xi$ for every limit ordinal $\eta\leq \alpha$. For every $\xi\in[0,\alpha)$ we take an arbitrary point $y_\xi\in f(U_{\xi+1}\setminus U_\xi)$. Now for every $x\in X$ we put
  $$
  f_n(x)=\left\{\begin{array}{ll}
              y_\xi, &  x\in V_{\xi+1}\setminus V_\xi,\\
              y_0,   & x\in X\setminus V_\alpha.
                \end{array}\right.
   $$
    Observe that $f_n:X\to \mathbb R$ is functionally $F_\sigma$-measurable, since the preimage $f^{-1}(W)$ of any open set $W\subseteq Y$ is an at most countable union of functionally $F_\sigma$-sets from the system $\{V_{\xi+1}\setminus V_\xi:\xi\in[0,\alpha)\}\cup\{X\setminus V_\alpha\}$. Remark~\ref{rem:local_conn_R_space} and~\cite[Theorem~4.1]{Karlova:EJM:2016} imply that $f_n\in {\rm B}_1(X,Y)$.

  It is easy to see that the sequence $(f_n)_{n=1}^\infty$ is uniformly convergent to $f$ on $E$. Now it follows from Lemma~\ref{lem:unif_conv_ext_R_space} that $f$ can be extended to a function $g\in {\rm B}_1(X,Y)$. According to \cite[Proposition 2]{Karlova:Mykhaylyuk:2017:ExtensionB1}, $g$ is functionally countably fragmented.
  \end{proof}

\section{Uniform approximation of bounded fragmented Baire-one functions}

A topological space $X$ is said to satisfy the {\it functionally discrete countable chain
condition (functionally DCCC)}, if every discrete collection of functionally open sets is at most countable.

\begin{proposition}\label{prop:sep_image}
  Let $X$ be a functionally DCCC space  and $Y$ be a topological space. If $f:X\to Y$ is a $\sigma$-strongly functionally discrete map, then  $f(X)$ has CCC.
\end{proposition}

\begin{proof}
   We take a $\sigma$-sfd base $\mathscr B=(\mathscr B_n:n\in\omega)$ for $f$ consisting of sfd families $\mathscr B_n$  in $X$.
    Notice that each family $\mathscr B_n$ is at most countable. Then $\mathscr B$ is also at most countable and let $\{B_k:k\in\omega\}$ be an enumeration of $\mathscr B$.

We consider  a disjoint family  $\mathscr G=(G_s:s\in S)$ of open sets in $f(X)$ and a map $\varphi:S\to 2^\omega$,
 $$
 \varphi(s)=\{k\in\omega: B_k\subseteq f^{-1}(G_s)\}.
 $$
Since $(\varphi(s):s\in S)$ is a family of mutually disjoint subsets of $\omega$, it is at most countable.
\end{proof}

\begin{theorem}\label{thm:approx_finite_valued}
  Let $X$ be a Lindel\"{o}ff space, $(Y,d)$ be a completely bounded metric space and \mbox{$f:X\to Y$} be a fragmented Baire-one map.
 Then there exists a sequence of finite-valued functionally countably fragmented maps   which is uniformly convergent to $f$ on $X$.
\end{theorem}

\begin{proof} Fix $m\in\omega$ and put $\varepsilon=\frac{1}{m}$. We take a finite $\varepsilon$-network $D=\{y_1,\dots,y_n\}$ in $Y$.
  Since $f\in{\rm B}_1(X,Z)$, for every $k\in\{1,\dots,n\}$ we use \cite[Lemma 4.2]{Karlova:CMUC:2013} and find functionally ambiguous sets $A_k$ and $B_k$ in $X$ such that
  \begin{gather}
    f^{-1}(B[y_k,\varepsilon])\subseteq A_k\subseteq f^{-1}(B(y_k,2\varepsilon)),\\
    f^{-1}(Y\setminus B(y_k,4\varepsilon))\subseteq B_k\subseteq f^{-1}(Y\setminus B[y_k,3\varepsilon]).\label{gath:i1}
  \end{gather}

  Fix $k\in \{1,\dots,n\}$ and a closed nonempty set $F\subseteq X$.   Since $f$ is $\varepsilon$-fragmented, there exists a relatively open nonempty set $U\subseteq F$ such that ${\rm diam}f(U)<\varepsilon$. The inequality
  \begin{gather*}
  d(f(A_k),f(B_k))\ge \varepsilon
  \end{gather*}
  implies that $U\cap A_k=\emptyset$ or $U\cap B_k=\emptyset$. Therefore, $\overline{F\cap A_k}\ne F\ne\overline{F\cap B_k}$. By Proposition~\ref{prop:SeparationTh} there exists a functionally countably resolvable set $C_k$ in $X$ such that $A_k\subseteq C_k\subseteq X\setminus B_k$.

  Notice that $(C_1,\dots,C_n)$ is a covering of $X$. For every $x\in X$ we put
  \begin{gather*}
    f_m(x)=\left\{\begin{array}{ll}
                  y_1, & x\in C_1, \\
                  y_k, & x\in C_k\setminus \bigcup_{j<k}C_j.
                \end{array}
    \right.
  \end{gather*}
  Proposition~\ref{prop:properties} implies that $f_m:X\to Y$ is functionally countably fragmented. Moreover, it follows from (\ref{gath:i1}) that
  \begin{gather*}
  d(f_m(x),f(x))<\frac{4}{m}
  \end{gather*}
for all $x\in X$.

  Hence, the sequence $(f_m)_{m=1}^\infty$ converges uniformly to $f$  on $X$.
\end{proof}

\section{Main results}
\begin{theorem}\label{thm:ext:main}
Any fragmented Baire-one function $f:X\to\mathbb R$ defined on a Lindel\"{o}ff space  is functionally countably fragmented.
\end{theorem}

\begin{proof}
Assume that $f$ is bounded. We apply Theorem~\ref{thm:approx_finite_valued} and obtain that $f$ is a uniform limit of  a sequence of functionally countably fragmented functions. Then $f$ is functionally countably fragmented by Proposition~\ref{prop:properties}~(\ref{prop:properties:it:4}).

Now let $f$ be an arbitrary Baire-one fragmented function. For every $x\in X$ we put $g(x)={\rm arctg}(f(x))$. It is easy to see that the function \mbox{$g:X\to (-\frac{\pi}{2},\frac{\pi}{2})$} is fragmented and Baire-one. As we have already proved, $g$ is functionally countably fragmented. We apply Proposition~\ref{prop:ext_countable_frag_ness_R} and extend $g$ to a Baire-one function $\tilde g:\beta X \to (-\frac{\pi}{2},\frac{\pi}{2})$. Then we consider a function $\tilde f:\beta X\to\mathbb R$,
$\tilde f(x)={\rm tg}(\tilde g(x))$. Since $\tilde f$ is a Baire-one function defined on a compact space $\beta X$, $\tilde f$ is functionally countably fragmented. Clearly, $\tilde f|_X=f$. It follows that $f$ is functionally countably fragmented function on $X$.
\end{proof}

\begin{theorem}\label{thm:equiv:main} Let $X$ be a Lindel\"{o}ff space and $f:X\to\mathbb R$ be a Baire-one function. Then following conditions are equivalent:
\begin{enumerate}
  \item $f$ is extendable to a Baire-one function on any completely regular superspace $Y\supseteq X$;

  \item $f$ is extendable to a Baire-one function on any compactification $Y$ of $X$;

  \item $f$ is extendable to a Baire-one function on $\beta X$;

    \item $f$ is functionally countably fragmented;

    \item $f$ is fragmented.
\end{enumerate}
\end{theorem}

\begin{proof}
Implications  1)~$\Rightarrow$~2)~$\Rightarrow$~3) and 4)~$\Rightarrow$~5) are obvious.

We prove 3)~$\Rightarrow$~4). Let $g\in{\rm B}_1(\beta X,\mathbb R)$ be an extension of $f$. Then $g$ is functionally countably fragmented by \cite[Proposition 2]{Karlova:Mykhaylyuk:2017:ExtensionB1}. Hence, $f=g|_X$ is functionally countably fragmented too.

Implication 5)~$\Rightarrow$~4) follows from Theorem~\ref{thm:ext:main}. Finally, Proposition~\ref{prop:ext_countable_frag_ness_R} implies that 4)~$\Rightarrow$~1).
\end{proof}

\section{Example}

In this section we construct a completely metrizable locally compact space $X$ and a Baire one function $f:X\to[0,1]$ which can not be extended to a Baire one function $g:\beta X\to [0,1]$. This gives a negative answer to \cite[Question 1]{Kal_Spu} (see also \cite[Theorem 7]{Karlova:Mykhaylyuk:2017:ExtensionB1}).

We start with two auxiliary assertions.

\begin{lemma}\label{l:3.1}
 Let $X$ be a metrizable separable space without isolated points and $A\subseteq X$ be a closed nowhere dense set. Then there exists a closed nowhere dense set $B$ in $X$ without isolated points, $A\subseteq B$ and $A$ is nowhere dense in $B$.
\end{lemma}

\begin{proof} Fix a base $(U_n:n\in\mathbb N)$ of the topology of $X$ and a point $x_0\in X$. Using the induction on $n$ it is easy to construct a sequence $(V_n)_{n=1}^{\infty}$ of nonempty open sets $V_n$ in $X$ and a sequence $(x_n)_{n=1}^{\infty}$ of points $x_n\in X$ with the properties:
\begin{enumerate}
\item[$(1)$] $V_n\subseteq U_n$ and $W_n=U_n\setminus\overline{V_n}\ne\emptyset$;

\item[$(2)$] $V_n\cap(A\cup\{x_k:0\leq k\leq n-1\})=\emptyset$;

\item[$(3)$] $x_n=x_0$, if $U_n\cap(A\cup\{x_k:0\leq k\leq n-1\})=\emptyset$;

\item[$(4)$] $x_n\in W_n\setminus(A\cup\{x_k:0\leq k\leq n-1\})$, if $U_n\cap(A\cup\{x_k:0\leq k\leq n-1\})\ne\emptyset$.
\end{enumerate}

It remains to put $B=X\setminus (\bigcup\limits_{n=1}^{\infty}V_n)$.
\end{proof}

\begin{lemma}\label{l:3.2}
 There exists a family $(F_t:t\in[0,1])$ of closed sets $F_t\subseteq [0,1]$ such that
\begin{enumerate}
 \item[$(1)$] $F_s$ is nowhere dense in $F_t$ for every $0\leq s<t\leq 1$;

\item[$(2)$] $F_t=\bigcap\limits_{s\in(t,1]}F_s$ for every $t\in[0,1)$.
 \end{enumerate}
\end{lemma}

\begin{proof} Let $\mathbb Q\cap[0,1]=\{r_0, r_1, r_2,\dots\}$,  $r_0=0$, $r_1=1$ and all $r_n$ are distinct. We put $A_0=\emptyset$, $A_1=[0,1]$. Now using the induction on $n$ and Lemma \ref{l:3.1} it is easy to construct a sequence $(A_n)_{n=0}^{\infty}$ of closed sets $A_n\subseteq[0,1]$ such that for every $n\in\mathbb N$ the set $A_n$ has no  isolated points and for all $n,k\geq 0$ with $r_n<r_k$ the set $A_n$ is nowhere dense in the set $A_k$. It remains to put
$F_t=\bigcap_{r_n< t} A_n$.
\end{proof}

 Let  $X$ be a topological space, $\varepsilon >0$ and let $f:X\to \mathbb R$ be an $\varepsilon$-fragmented function. The smallest ordinal $\alpha$ such that there exists a strictly increasing sequence
$$
\emptyset=U_0\subset U_1\subset\dots \subset U_\alpha=X
$$
of open sets $U_\xi$ such that
 \begin{itemize}

\item[(1)]  $U_\beta =\bigcup\limits_{\xi<\beta}U_\xi$ for every limit ordinal $\beta\leq\alpha$;

\item[(2)]  ${\rm diam}(f(U_\xi\setminus\bigcup\limits_{\eta<\xi}U_\eta))<\varepsilon$ for every $\xi\leq\alpha$;
  \end{itemize}
is called {\it an index of $\varepsilon$-fragmentability of $f$.

\begin{proposition}\label{pr:3.3}
 For every ordinal $\alpha\in[\omega,\omega_1)$ there exists a function $f:[0,1]\to[0,1]$ such that the index of the $1$-fragmentability of $f$ is equals to $\alpha+1$.
\end{proposition}

\begin{proof} According to Lemma \ref{l:3.2} there exists a strictly decreasing sequence \mbox{$(F_\xi:0\leq \xi \leq \alpha)$} of closed sets $F_\xi\subseteq[0,1]$ such that

$(1)$\,\,$F_0=[0,1]$;

$(2)$\,\,$F_\xi$ is nowhere dense in $F_\eta$ for every $0\leq\eta <\xi\leq \alpha$;

$(3)$\,\,$F_\xi=\bigcap\limits_{0\leq \eta<\xi}F_\eta$ for every $\xi\leq\alpha$.

Let $\Lambda$ be the set of all even ordinals $\xi\leq \alpha$. Now we put $F_{\alpha+1}=\emptyset$, $A=\bigcup\limits_{\xi \in \Lambda}(F_\xi\setminus F_{\xi+1})$ and consider the function $f:[0,1]\to[0,1]$, $f=\chi|_A$.

For every $\xi\leq \alpha+1$ we put $G_\xi=[0,1]\setminus F_\xi$. Note that ${\rm diam}f(G_{\xi+1}\setminus G_\xi)=0$ for every $\xi\leq\alpha$ and $G_{\alpha+1}=[0,1]$.
Therefore, the index of the fragmentability of $f$ is not greater than $\alpha+1$.

Now let $(U_\xi:0\leq\xi\leq\beta)$ be a regular covering of $[0,1]$ such that ${\rm diam}f(U_{\xi+1}\setminus U_\xi)<1$ for every $\xi<\beta$. It easy to see that $U_\xi\subseteq G_\xi$ for every $\xi\leq \alpha+1$. Therefore, $\beta\geq \alpha+1$.
\end{proof}

\begin{theorem}\label{th:3.4}
There exist a completely metrizable locally compact space $X$ and a Baire one function $f:X\to[0,1]$ such that $f$ is not countably fragmented, in particular, $f$ can not be extended to a Baire one function $g:\beta X\to[0,1]$.
\end{theorem}

\begin{proof} For every $\alpha<\omega_1$ we put $X_\alpha=[0,1]$ and consider the completely metrizable locally compact space $X=\bigoplus\limits_{\alpha<\omega_1}X_\alpha$. Using Proposition \ref{pr:3.3} for every $\alpha<\omega_1$ we choose a countably fragmented function $f_\alpha:X_\alpha\to[0,1]$ such that the index of fragmentability of $f$ is greater than $\alpha$. Now we consider the function $f:X\to[0,1]$, $f(x)=f_\alpha(x)$ if $x\in X_\alpha$. Since every $f_\alpha$ is a Baire one function, $f$ is a Baire one function too. Moreover, it is clear that $f$ is not countably fragmented.
\end{proof}

\end{document}